\documentclass{amsart}

\usepackage[utf8]{inputenc}

\usepackage{amssymb}
\usepackage[hidelinks]{hyperref}
\usepackage[textsize=small]{todonotes}
\usepackage{tikz}
\usetikzlibrary{cd}
\usetikzlibrary{decorations.text}
\usepackage{float}
\usepackage{xcolor}
\usepackage[normalem]{ulem}
\usepackage{mathrsfs}
\usepackage{graphicx} 

\theoremstyle{plain}
\newtheorem{theorem}{Theorem}[section]
\newtheorem*{theorem-wo-number}{Theorem}
\newtheorem{lemma}[theorem]{Lemma}

\newtheorem{proposition}[theorem]{Proposition}

\theoremstyle{definition}
\newtheorem{definition}[theorem]{Definition}

\newtheorem{example}[theorem]{Example}

\newtheorem{question}[theorem]{Question}

\theoremstyle{remark}

\newtheorem*{claim}{Claim}

\newcommand{\bQ}{\mathbb{Q}}
\newcommand{\Q}{\bQ}
\newcommand{\bN}{\mathbb{N}}
\newcommand{\N}{\bN}

\newcommand{\cA}{\mathcal{A}}
\newcommand{\cB}{\mathcal{B}}

\newcommand{\cF}{\mathcal{F}}

\newcommand{\cG}{\mathcal{G}}

\newcommand{\cI}{\mathcal{I}}
\newcommand{\I}{\cI}
\newcommand{\cJ}{\mathcal{J}}
\newcommand{\J}{\cJ}

\newcommand{\cP}{\mathcal{P}}

\newcommand{\continuum}{\mathfrak{c}}

\newcommand{\bnumber}{\mathfrak{b}}

\newcommand{\dnumber}{\mathfrak{d}}

\DeclareMathOperator{\add}{add}
\DeclareMathOperator{\non}{non}
\DeclareMathOperator{\cov}{cov}
\DeclareMathOperator{\cof}{cof}
\DeclareMathOperator{\adds}{\add^*}
\DeclareMathOperator{\addo}{\add^*_{\omega}}
\DeclareMathOperator{\nons}{\non^*}
\DeclareMathOperator{\nono}{\non^*_{\omega}}
\DeclareMathOperator{\covs}{\cov^*}
\DeclareMathOperator{\covp}{\cov^*_+}
\DeclareMathOperator{\cofs}{\cof^*}
\DeclareMathOperator{\cofo}{\cof^*_{\omega}}
\newcommand{\FIN}{\mathsf{Fin}} 
\newcommand{\fin}{\FIN}
\newcommand{\Fin}{\FIN}
\newcommand{\finxfin}{\FIN\otimes\FIN}

\newcommand{\conv}{\mathsf{conv}} 
\newcommand{\BI}{\mathsf{BI}} 

\DeclareMathOperator{\closure}{cl}

\renewcommand{\subset}{\subseteq}
\renewcommand{\supset}{\supseteq}

\begin{document}


\title[Infima and cardinal characteristics of critical ideals]{Infima and cardinal characteristics of critical ideals for countable compact spaces}


\author{Ma\l{}gorzata Kowalczuk}
\address{Institute of Mathematics\\ Faculty of Mathematics, Physics and Informatics\\ University of Gda\'{n}sk\\ ul.~Wita Stwosza 57\\ 80-308 Gda\'{n}sk\\ Poland}
\email{m.kowalczuk.090@studms.ug.edu.pl}


\date{\today}


\subjclass[2020]{
Primary: 03E17, 03E05, 03E15}

%
%
%


\keywords{
ideal, 
cardinal characteristics of an ideal, 
conv ideal, 
Kat\'{e}tov order,
Borel complexity of an ideal, 
compact countable space.
}


\begin{abstract}
For each countable ordinal $\alpha \ge 2$, the ideals $\conv_\alpha$ were introduced in 
\emph{Critical ideals for countable compact spaces} (to appear in Fund. Math., see also  
arXiv:2503.12571)
to characterize compact countable spaces homeomorphic to $\omega^\alpha \cdot n+1$ with the order topology. We study the structure of these ideals in the Kat\v{e}tov order, namely for limit ordinals $\alpha$, we show that  $\conv_{\alpha}$ do not serve as greatest lower bounds of the $\conv_\beta$ for $\beta<\alpha$. We therefore define the ideals $\conv_{<\alpha}$ with this property and show that together, the ideals $\conv_\alpha$ and $\conv_{<\alpha}$ form intertwined decreasing hierarchies of $\Sigma^0_4$- and $\Pi^0_5$-complete ideals. Furthermore, we examine several cardinal invariants of $\conv_\alpha$, computing invariants that have recently appeared in the literature in various contexts.

\end{abstract}


\maketitle




\section{Introduction}\label{sec:intro}

See Preliminaries for notions and notations used in the introduction.

For each countable ordinal $\alpha\geq 2$, the ideal $\conv_\alpha$ was introduced in \cite{critical-ideals-RF-MK-AK} and used to characterize the class of all compact countable spaces homeomorphic to $\omega^\alpha \cdot n+1$ with the order topology. In that work, the characterization was formulated in terms of the existence of a convergent subsequence defined on a set not belonging to $\conv_\alpha$. 
These ideals extend $\finxfin$ and $\BI$ in the sense that $\conv_2$ is isomorphic to $\finxfin$ and $\conv_3$ is isomorphic to $\BI$ , and the resulting hierarchy proceeds downward in the Kat\v{e}tov order. The ideal $\BI$ was introduced in \cite{MR4584767} and is isomorphic to the ideal $\mathcal{WT}$ introduced in \cite{BRENDLE-CJM-2025}, where the authors showed its connection with weakly tight almost disjoint families.

In this paper, we investigate the ideals $\conv_\alpha$, examining both their structure and some of their properties.

In the first section we focus on the structure.
The ideals $\conv_\alpha$, for all countable ordinals $\alpha \geq 2$, form a decreasing hierarchy in the Kat\v{e}tov order. We show, however, that for a limit ordinal $\alpha$, the ideal $\conv_\alpha$ is not the greatest lower bound of the ideals $\conv_\beta$ for $\beta<\alpha$.
We therefore define an ideal $\conv_{<\alpha}$ which serves as the greatest lower bound of  $\{\conv_\beta : \beta<\alpha\}$ and show that it is strictly above $\conv_\alpha$ in the Kat\v{e}tov order. Moreover, we prove that the ideals $\conv_{<\alpha}$ are $\Pi^0_5$-complete. This result gains additional significance when contrasted with the fact that the ideals $\conv_\alpha$ are $\Sigma^0_4$-complete. Consequently, we obtain two intertwined chains of ideals of length $\omega_1$ with different Borel complexities.

In the second section of this paper, we focus on some cardinal invariants, which are fundamental tools for studying ideals on countable sets. They were originally introduced in \cite{MR1686797} in the context of studying ultrafilters on $\omega$. More recently, they have been independently rediscovered by other authors and applied in a variety of different contexts. In \cite{Alek-Miller-Nulls,Alek-Arturo-Miller-Trees}, these invariants were used to study ideals related to Laver and Miller forcings. In \cite{MR4905423}, to study connections between different types of convergences of real-valued functions. We also compute the invariant $\covp(\I)$ introduced in \cite{MR4472520}, in which the authors investigate its relationship to a form of destructibility of ideals in forcing extensions.
 We calculate these new invariants for the ideals that are the main focus of this paper, namely the ideals $\conv_\alpha$.


\section{Preliminaries}
\label{sec:prelim}

Recall that an ordinal number $\alpha$ is equal to the set of all ordinal numbers less than $\alpha$. 
In particular, the  smallest infinite ordinal number $\omega=\{0,1,\dots\}$ is equal to the set of all natural numbers $\N$, and each natural number $n = \{0,\dots,n-1\}$ is equal  to the set of all natural numbers less than $n$.
Using this identification, we can for instance write $n\in k$ instead of $n<k$ and $n<\omega$ instead of $n\in \omega$. 
Moreover, for ordinals $\alpha<\beta$, we write 
$[\alpha,\beta]$ to denote the set of all ordinals $\xi$ such that  $\alpha\leq \xi\leq \beta$, and similarly for $(\alpha,\beta)$ and other intervals.
 We write  $[A]^\omega$ to denote the family of all infinite countable subsets of $A$.
 By $A\subseteq^* B$ we mean that $A\setminus B$ is finite, and by $f\leq^* g$ for functions $f,g:\omega\to\omega$ we mean that there is some $k\in\omega $ such that $f(n)\leq g(n)$ for all $n\geq k$.

An \emph{ideal} on a nonempty set $X$ is a  family $\I\subseteq\cP(X)$ that satisfies the following properties:
\begin{enumerate}
\item $\emptyset\in \I$ and $X\not\in\I$,
\item if $A,B\in \I$ then $A\cup B\in\I$,
\item if $A\subseteq B$ and $B\in\I$ then $A\in\I$,
\item $\I$ contains all finite subsets of $X$. 
\end{enumerate}
We write $\I^+=\{A\subseteq X: A\notin\I\}$ and call it the \emph{coideal of $\I$}.
An ideal $\I$ is \emph{tall}  if for every infinite $A\subseteq X$ there is an infinite $B\in\I$ such that $B\subseteq A$.

 Let $\I$ and $\J$ be ideals on $X$ and $Y$ respectively.
We say that $\I$ and $\J$ are \emph{isomorphic} (in short $\I\approx\J$) 
if there exists a bijection $f:X\rightarrow Y$ such that $A\in\I\iff f[A]\in\J$ for every $A\subset X$.
We say that \emph{$\J$ is below $\I$ in  the Kat\v{e}tov order} (in short $\J\leq_{K}\I$) if there is a function $f:X\rightarrow Y$ such that $f^{-1}[A]\in\I$ for every $A\in \J$. 
We say that the ideals $\I$ and $\J$ are \emph{$\leq_K$-equivalent} if $\I\leq_K\J$ and $\J\leq_K\I$. 
We say that an ideal $\J$ on $Y$  is below an ideal $\I$ on $X$  in  the \emph{Kat\v{e}tov-Blass order} (in short $\J\leq_{KB}\I$) if there is a finite-to-one function $f:X\to Y$ such that $f^{-1}[A]\in\I$ for every $A\in \I$. 
The \emph{disjoint union} of $\I$ and $\J$ is an ideal on $(\{0\}\times X)\cup(\{1\}\times Y)$ given by
$\I \oplus \J =\{A\subset (X\times\{0\})\cup(Y\times \{1\}): \{x\in X: (x,0)\in A\}  \in \I \text{ and } \{y\in Y: (y,1)\in A\} \in \J\}.$

By identifying sets of natural numbers with their characteristic functions,
we equip $\cP(\omega)$ with the topology of the Cantor space $\{0,1\}^\omega$ and therefore
we can assign topological complexity to ideals on $\omega$.
In particular, an ideal $\I$ is Borel ($F_\sigma$, analytic, resp.) if $\I$ is Borel ($F_\sigma$, analytic, resp.) as a subset of the Cantor space.
We will also be interested in $\Sigma^0_\alpha$-complete sets and $\Pi^0_\alpha$-complete sets, for some $\alpha\in\omega_1$. These notions are thoroughly explained in \cite{MR1321597}.

\begin{example}\  
\begin{enumerate}

 \item 
     $\fin(X)$ is  the ideal of all finite subsets of $X$.
    We  write $\fin$ instead of $\fin(\omega)$.

\item $\{\emptyset\} \otimes \Fin $ is the ideal on $\omega \times\omega$ defined by $$A\in \{\emptyset\}\otimes \fin \iff \forall i\in\omega \, (|\{j: (i,j)\in A\}|<\omega). $$

\item 
$\Fin^2= \Fin\otimes\Fin$ is the ideal on $\omega\times \omega$ defined by
\begin{equation*}
A\in \Fin^2 \iff \exists i_0\in\omega\, \forall i\geq i_0\, (|\{j\in\omega: (i,j)\in A\}|<\omega).
\end{equation*}

\item $\BI$ is the ideal  on $ \omega\times \omega \times\omega$
introduced in \cite[Definition~4.1]{MR4584767}
and defined by 
\begin{equation*}
    \begin{split}
A\in \BI 
\iff 
\exists i_0\in\omega\, \left[\right.
&
\forall i< i_0\, (\, \{(j,k)\in\omega\times\omega : (i,j,k)\in A\}\in \Fin^2)
\  \land 
\\&
\left.
\hspace{-3pt}\forall i\geq i_0\, (|\{(j,k)\in\omega\times\omega : (i,j,k)\in A\}|<\omega)
\right].
    \end{split}
\end{equation*}

\end{enumerate}
 \end{example}

For a  topological space $X$, 
we write $c(X)$ to denote the set of all convergent sequences in $X$.
For a subset $D\subseteq X$ of a topological space $X$, we write  $\conv(D)$ to denote 
the ideal on  $D$ consisting of all subsets of $D$ which can be covered by ranges of finitely many sequences in $D$ which are convergent in $X$ i.e.
$A\in \conv(D)$
if and only if $A\subseteq D$ and there exist $k\in \omega$ and sequences $(d^{(i)}_n)_{n\in \omega}\in c(X)\cap D^\omega$ for $i<k$ such that 
$$A \subseteq \bigcup_{i<k}\left\{d^{(i)}_n:n\in \omega\right\}.$$

A point $p\in X$ is an \emph{accumulation} (a.k.a.~\emph{limit}) \emph{point} of a set $A\subseteq X$ in a topological space $X$ if $p\in \closure(A\setminus\{p\})$.
By $A^d$ we denote the \emph{derived set of $A$} i.e.~the set of  all accumulation  points of $A$.

    Let $X$ be a sequentially compact space.
    Let $D\subseteq X$ be a countable infinite subset of $X$.
  For every $A\subseteq D$, 
$$A\in \conv(D) \iff A^{d} \text{ is finite}.$$

For every infinite set $A\subseteq D$ there exists an infinite set $B\subseteq A$ such that  $B\in \conv(D)$. In particular, the ideal  $\conv(D)$ is tall.\label{lem:conv-ideal-via-derivative-set:item-4}

\begin{definition}
\ 
\begin{enumerate}
    \item 
If $X=[0,1]$ with the Euclidean topology, we define the ideal
$$\conv = \conv(\Q\cap [0,1]).$$
\item For a countable  ordinal $\alpha\geq 2$ and 
the space $X=\omega^{\alpha}+1$ with the order topology, we define the ideal
$$\conv_{\alpha} = \conv(\omega^{\alpha}+1).$$
In the rest of the paper, we write ``$\alpha$'' to mean ''a countable ordinal  $\alpha\geq 2$''.
\end{enumerate}
\end{definition}
The ideals $\conv_\alpha$ for all $\alpha$ 
are $\Sigma^0_4$-complete 
(see \cite[Corollary~7.4]{critical-ideals-RF-MK-AK})
and the ideal $\conv$ is $\Sigma^0_4$-complete (see \cite[Proposition~4.9]{MR4584767} and \cite[p.~21]{alcantara-phd-thesis}).

It is worth noting that whenever $X\subset \omega^\alpha+1$ and $Y\subset \omega^\alpha+1$ are order isomorphic for some $\alpha$, then the ideal $\conv(X)$ is isomorphic to the ideal $\conv(Y)$.
We will frequently use the above observation together with the fact that $\omega^\alpha$ is order isomorphic to every interval of the form $[\omega^\beta, \omega^\alpha)$ for some $\beta<\alpha$ (see e.g.~\cite[Exercise 5 in Chapter I]{MR597342}).

\begin{proposition}[\cite{critical-ideals-RF-MK-AK}]
\label{prop:gwiazdka}
Let  $A\subseteq \omega^{\alpha}+1$.
The following conditions are equivalent.
\begin{enumerate}
    \item $A\in \conv_\alpha$.
    \item For every increasing sequence $(\lambda_n)_{n<\omega}$ in $\omega^\alpha$, the set
    $A\cap [\lambda_n, \lambda_{n+1}]$
is finite for all but finitely many $n$.
\end{enumerate}\end{proposition}

As a corollary, we obtain that the ideals $\conv_2$ and $\fin\otimes \fin$ ($\conv_3$ and $\BI$, resp.) are isomorphic.

The following theorem establishes the structure of the aforementioned ideals in the Kat\v{e}tov order and additionally records one of their properties.

\begin{proposition}[\cite{critical-ideals-RF-MK-AK}] \ 
\label{prop:hasldkfhas}
\begin{enumerate}
    \item If $\alpha<\beta$, then $\conv_\beta\leq_K\conv_\alpha$  and $\conv_\alpha\not\leq_K\conv_\beta$.     
    \item  $\conv \leq_K \conv_{\alpha}$ and  $\conv_\alpha\not\leq_K\conv$ for every  $\alpha$.
    \item $\conv_\alpha\leq_K \I \iff \conv_\alpha \leq_{KB} \I$ for every $\alpha $ and every ideal $\I$.
\end{enumerate}
 
\end{proposition}


\section{Infima of critical ideals for countable compact spaces}
In this section, for each limit ordinal $\alpha$, we define an ideal  $\conv_{<\alpha}$ and prove that it is strictly above the ideal $\conv_\alpha$ in the Kat\v{e}tov order and that it is the greatest lower bound of the family $\{\conv_\beta : \beta <\alpha\}$ with respect to the Kat\v{e}tov order. Then we investigate its Borel complexity.

\begin{definition}
For every limit ordinal $\alpha $ define the ideal, 
    $$\conv_{<\alpha} =\{A\subset \omega^{\alpha}+1 : \forall\beta<\alpha \ (A\cap \omega^\beta \in\conv_\beta) \}. $$
\end{definition}

Observe that the ideal $\conv_{<\alpha}$ is Kat\v{e}tov equivalent to the disjoint union of $\{\conv_\beta : \beta <\alpha\}$, which serves as a natural greatest lower bound in the Kat\v{e}tov order. The crucial point, however, is that $\conv_{<\alpha}$ is not Kat\v{e}tov equivalent to $\conv_\alpha$. This is easier to prove when working directly with $\conv_{<\alpha}$, since both ideals are defined on the same space.

\begin{proposition}
Let $\alpha$ be a  limit ordinal.
    \begin{enumerate}
        \item The ideal $\conv_{<\alpha}$ is the greatest lower bound of the family $\{\conv_\beta : \beta <\alpha\}$ in the Kat\v{e}tov order i.e.
            \begin{enumerate}
\item  $ \conv_{<\alpha}\leq_K \conv_\beta$ for every $\beta<\alpha$,

 \item 
 for all ideals $\I$, whenever  $\I\leq_K\conv_\beta$ for each $\beta<\alpha$, then $\I\leq_K \conv_{<\alpha}.$
           \end{enumerate}

        \item 
        $\conv_\alpha\leq_K\conv_{<\alpha}$
        and 
      $\conv_{<\alpha}\nleq_K \conv_\alpha$
    
    \end{enumerate}
\end{proposition}

\begin{proof}
    (1a) It is witnessed by the map $id:\omega^{\beta}+1\to\omega^\alpha+1$, $id(\xi)=\xi$ for each $\xi\in\omega^\beta+1$. 

    (1b) Take a strictly increasing sequence $(\alpha_n)$ that is cofinal in $\alpha$ with $\alpha_0\geq 2$ and let $f_n :\omega^{\alpha_n}+1\to\omega$ be the witness for $\I\leq_K \conv_{\alpha_n}$ for each $n$. Define a function $f:\omega^\alpha+1\to\omega$ by $$f=f_0\cup \bigcup_{n>0} f_n \restriction (\omega^{\alpha_{n-1}}, \omega^{\alpha_n}]\cup \{(\omega^\alpha,0)\}.$$ Then $f$ witnesses $\I\leq_K\conv_{<\alpha}$. To see this, take any $A\notin\conv_{<\alpha}$. Then there is some $\beta<\alpha$ with   $A\notin\conv_{\beta}$. Let $n$ be the smallest natural number with $\alpha_n\geq\beta$. Then $A\notin\conv_{\alpha_n}$. Since  $\omega^{\alpha_n}+1=[0,\omega^{\alpha_0}]\cup\bigcup_{0<i\leq n}(\omega^{\alpha_{i-1}}, \omega^{\alpha_i}]$, there is some $k\in \omega $ with  $0<k\leq n$ such that $A\cap(\omega^{\alpha_{k-1}}, \omega^{\alpha_k}] \notin \conv_{\alpha_k} $. And since $f_k$ is a witness for $\I\leq_K \conv_{\alpha_k}$, we get that $f[A] \supseteq f[A\cap(\omega^{\alpha_{k-1}}, \omega^{\alpha_k}]]=f_k[A\cap(\omega^{\alpha_{k-1}}, \omega^{\alpha_k}]]\notin\I$.

    (2) 
     The inequality $\conv_\alpha\leq_K\conv_{<\alpha}$ follows from item (1b) and Proposition~\ref{prop:hasldkfhas}.
    Below, we prove that the reversed inequality does not hold.
    Suppose for the sake of contradiction that there is some $f : \omega ^\alpha + 1 \to \omega ^\alpha + 1$ such that $f^{-1}[B]\in \conv_\alpha$ for every $B\in\conv_{<\alpha}$. Take a strictly increasing sequence $(\alpha_n)$ that is cofinal in $\alpha$ with $\alpha_0=2$. 
    We will construct sequences $(k_i)_{i\geq 1}$, $(n_i)_{i\in\omega}$, and $(x^i_n)_{n\in\omega}$, $(y^i_n)_{n\in\omega}$ for $i\in\omega$ with the following properties:
    \begin{itemize}
        \item the sequences $(k_i)_{i\geq 1}$ and $(n_i)_{i\in\omega}$ are strictly increasing and $n_i<k_{i+1}$ for each $i\in\omega$,
        \item $\{x^0_n:n\in\omega\}\subset [0,\omega^2]$ and $\{x^i_n:n\in\omega\}\subset [\omega^{\alpha_{k_i}},\omega^{\alpha_{k_i}} + \omega^{\alpha_{n_{i-1}}})$ for each $i\geq 1$,
         \item $\{y^0_n:n\in\omega\}\subset [0,\omega^{\alpha_{n_0}})$ and $\{y^i_n:n\in\omega\}\subset [\omega^{\alpha_{n_{i-1}}},\omega^{\alpha_{n_i}})$ for each $i\geq 1$, \item $(x^i_n)_n$ and $(y^i_n)_n$ are convergent, and $f(x^i_n)=y^i_n$ for all $n$ and $i$.
    \end{itemize}
Suppose we have constructed the mentioned sequences. Let $A=\{x^i_n : i,n\in\omega\}$ and $B=\{y^i_n : i,n\in\omega\}$. Then $A\notin \conv_\alpha$ and $B\in\conv_{<\alpha}$. However, $A\subset f^{-1}[B]$ and since $f$ is the Kat\v{e}tov reduction $f^{-1}[B]\in\conv_\alpha$, a contradiction.

To finish the proof of item (2), we need to construct $k_i$, $n_i$, and $(x^i_n)$, $(y^i_n)$. 
We proceed by induction  on $i\in\omega$. For $i=0$, note that  since $[0,\omega^2]\notin\conv_\alpha$ and $f$ is the witness for the Kat\v{e}tov inequality, $f[[0,\omega^2]]\notin\conv_{<\alpha}$, so there is some $n_0$ such that $f[[0,\omega^2]]\cap\omega^{\alpha_{n_0}}\notin\conv_{\alpha_{n_0}}$. Hence there is an injective convergent sequence $(y^0_n)_n$ with values in the set $ f[[0,\omega^2]]\cap\omega^{n_0}$.
We can choose an injective sequence $(x^0_n)_n $ with values in the interval $[0,\omega^2]$ such that $f(x^0_n)=y^0_n$ for each $n$.
Without loss of generality,  we can assume that the sequence  $(x^0_n)_n$ is convergent and that $(y^0_n)_n$ remains convergent. 

Suppose we have already constructed $k_i, n_i, (x^i_n)_{n\in\omega}$ and $(y^i_n)_{n\in\omega}$ for some $i$ that satisfy the conditions above.

\begin{claim}    
There is some $k_{i+1}>\max(n_i,k_i)$ such that the set $$(f[[\omega^{\alpha_{k_{i+1}}},\omega^{\alpha_{k_{i+1}}}+\omega^{\alpha_{n_i}}]])^d\setminus (\omega^{\alpha_{n_i}}+1)$$ is infinite.
\end{claim}

\begin{proof}[Proof of Claim]
Suppose for the sake of contradiction that the set $(f[[\omega^{\alpha_k},\omega^{\alpha_k}+\omega^{\alpha_{n_i}}]])^d\setminus (\omega^{\alpha_{n_i}}+1)$ is finite for every  $k>\max(n_i,k_i)$.
For every such $k$, define the set $$F_k=f[[\omega^{\alpha_k},\omega^{\alpha_k}+\omega^{\alpha_{n_i}}]]\setminus (\omega^{\alpha_{n_i}}+1)$$ and note that since $(F_k)^d$ is finite, $F_k\in\conv_{<\alpha}$, and therefore $f^{-1}[F_k]\in\conv_\alpha$.  Let $g: \bigcup_{k>\max(n_i,k_i)} [\omega^{\alpha_k},\omega^{\alpha_k}+\omega^{\alpha_{n_i}}]\to\omega^{\alpha_{n_i}}+1$ be a function such that 

\begin{enumerate}
    \item $g$ coincides with $f$ on the set
$\bigcup_{k>\max(n_i,k_i)} [\omega^{\alpha_k},
 \omega^{\alpha_k}+\omega^{\alpha_{n_i}}]\setminus f^{-1}[F_k]$,

 \item $g\restriction ([\omega^{\alpha_k},\omega^{\alpha_k}+\omega^{\alpha_{n_i}}]\cap f^{-1}[F_k])$ is injective for each $k>\max(n_i,k_i)$,
 \item $g[ [\omega^{\alpha_k},\omega^{\alpha_k}+\omega^{\alpha_{n_i}}]\cap f^{-1}[F_k]]\subset[\omega\cdot k, \omega\cdot (k+1)).$ 

 \end{enumerate}

We will show that $g $ witnesses $\conv_{\alpha_{n_i}}\leq_K \conv(\bigcup_{k>\max(n_i,k_i)} [\omega^{\alpha_k},\omega^{\alpha_k}+\omega^{\alpha_{n_i}}])$ which will give us a contradiction since 
$$\conv\left(\bigcup_{k>\max(n_i,k_i)} [\omega^{\alpha_k},\omega^{\alpha_k}+\omega^{\alpha_{n_i}}]\right) \approx \conv_{\alpha_{n_{i}}+1}\not\geq_K\conv_{\alpha_{n_i}}.$$

To see that the function $g$ witnesses the Kat\v{e}tov  reduction, take any $A\notin  \conv(\bigcup_{k>\max(n_i,k_i)} [\omega^{\alpha_k},\omega^{\alpha_k}+\omega^{\alpha_{n_i}}])$. Since $A^d$ is infinite we have two cases.

\emph{Case 1.} There is  $k>\max(n_i,k_i)$ such that $A\cap [\omega^{\alpha_k},\omega^{\alpha_k}+\omega^{\alpha_{n_i}}]\notin \conv([\omega^{\alpha_k},\omega^{\alpha_k}+\omega^{\alpha_{n_i}}])$. Then $B=A\cap [\omega^{\alpha_k},\omega^{\alpha_k}+\omega^{\alpha_{n_i}}]\setminus f^{-1} [F_k]\notin \conv([\omega^{\alpha_k},\omega^{\alpha_k}+\omega^{\alpha_{n_i}}]) $, so $g[B]=f[B]\notin \conv_{<\alpha}$. However, $g[B]\subset [0,\omega^{\alpha_{n_i}}]$ and therefore $g[B]\notin\conv_{\alpha_{n_i}}.$

\emph{Case 2.} The set $\{l>\max(n_i,k_i) : |A\cap [\omega^{\alpha_l}, \omega^{\alpha_l}+\omega^{\alpha_{n_i}}]|=\omega\}$ is infinite.
Let $\{l_j :j\in\omega\}$ be an
enumeration of the set above, and let $B_l=A\cap [\omega^{\alpha_l}, \omega^{\alpha_l}+\omega^{\alpha_{n_i}}]$ for every $l$.
Then
$$g\left[\bigcup_{j\in\omega} B_{l_j}\right]=f\left[\bigcup_{j\in\omega} B_{l_j}\setminus f^{-1}[F_{l_j}]\right]\cup g\left[\bigcup_{j\in\omega} B_{l_j}\cap f^{-1}[F_{l_j} ]\right] .$$
Since $\bigcup_{j\in\omega} B_{l_j}\notin\conv_\alpha$, either $\bigcup_{j\in\omega} B_{l_j}\setminus f^{-1}[F_{l_j}]\notin\conv_\alpha$ or  $\bigcup_{j\in\omega} B_{l_j}\cap f^{-1}[F_{l_j} ]\not\in\conv_\alpha$. In the former case, it follows that $f[\bigcup_{j\in\omega} B_{l_j}\setminus f^{-1}[F_{l_j}]]\notin\conv_{<\alpha}$. But $f[\bigcup_{j\in\omega} B_{l_j}\setminus f^{-1}[F_{l_j}]]\subset [0,\omega^{\alpha_{n_i}}]$ and therefore $g[A]\supseteq f[\bigcup_{j\in\omega} B_{l_j}\setminus f^{-1}[F_{l_j}]]\notin\conv_{\alpha_{n_i}}$. In the latter case, let us note that since $f^{-1}[F_{l_j}]\in\conv_\alpha$ for each $j$, the set $\{j\in\omega : |B_{l_j}\cap f^{-1}[F_{l_j}]|=\omega\}$ is infinite. By the definition of $g$, we get that $ g[\bigcup_{j\in\omega} B_{l_j}\cap f^{-1}[F_{l_j} ]] \notin\conv_{\alpha_{n_i}}$.
\end{proof}
Let  $k_{i+1}$ be as  stated in the claim. Then there is some $n_{i+1}$ and an injective convergent sequence $(y^{i+1}_n)_n$ with $$\{y^{i+1}_n:n\in\omega\}\subset f[[\omega^{\alpha_{k_{i+1}}},\omega^{\alpha_{k_{i+1}}}+\omega^{\alpha_{n_i}}]\cap (\omega^{\alpha_{n_i}},\omega^{\alpha_{n_{i+1}}}). $$ 
 Let $ (x^{i+1}_n)_n$ be a sequence such that 
 $$\{x^{i+1}_n:n\in\omega \} \subset [\omega^{\alpha_{k_{i+1}}},\omega^{\alpha_{k_{i+1}}}+\omega^{\alpha_{n_i}}] \quad  \text{and} \quad f(x^{i+1}_n)=y^{i+1}_n$$ for each $n$. Without loss of generality, we can assume it is injective and covergent. Hence we are done with the construction.
\end{proof}

\begin{proposition}
   The ideal $\conv_{<\alpha}$ is $\Pi^0_5$-complete
   for every limit ordinal $\alpha$.
\end{proposition}

\begin{proof}
To see that the ideal $\conv_{<\alpha}$ is $\Pi^0_5$, note that $$\conv_{<\alpha}=\bigcap_{\beta\in\alpha}\{A\subset\omega^\alpha+1 : A\cap\omega^{\beta}\in\conv_{\beta}\},$$
where the set $C_\beta=\{A\subset\omega^\alpha+1 : A\cap\omega^{\beta}\in\conv_{\beta}\}$ is $\Sigma^0_4$ for each $\beta$, since $\conv_\beta\in\Sigma^0_4$ and the set $C_\beta$ is homeomorphic to the set $\conv_\beta \times \cP((\omega^\beta,  \omega^\alpha])$.

To show completeness, take any $A\in\Pi^0_5(X)$, where $X$ is a zero-dimensional Polish space. We need to find a continuous function $f: X\to \cP(\omega^\alpha+1)$ with $f^{-1}[\conv_{<\alpha}]=A$. Take a strictly increasing sequence $(\alpha_k)_k$ cofinal in $\alpha$ with $\alpha_0=0$ and $\alpha_1\geq 2$. Since $A=\bigcap \{A_n :n\in\omega\}$ for some $A_n\in\Sigma^0_4(X)$ and $\conv([\omega^{\alpha_n},\omega^{\alpha_{n+1}}))$ is $\Sigma^0_4$-complete (as it is isomorphic to $\conv_{\alpha_{n+1}})$,  there exists a continuous function $f_n:X\to \cP([\omega^{\alpha_n},\omega^{\alpha_{n+1}}))$ with $f_n^{-1}[\conv([\omega^{\alpha_n},\omega^{\alpha_{n+1}}))]=A_n $
for each $n$. Define $f:X\to \cP(\omega^\alpha+1)$ by $$f(x)=\bigcup_n f_n(x).$$ Then $f$ is continuous and 
\begin{equation*}
    \begin{split}   
f^{-1}[\conv_{<\alpha}]
=& 
\bigcap_{n\in\omega}\left\{x\in X : f(x)\cap[\omega^{\alpha_n},\omega^{\alpha_{n+1}})\in \conv([\omega^{\alpha_n},\omega^{\alpha_{n+1}}))\right\}
\\=&
\bigcap_{n\in\omega}\left\{x\in X : f_n(x)\in \conv([\omega^{\alpha_n},\omega^{\alpha_{n+1}}))\right\}
=\bigcap_{n\in\omega}A_n=A.
   \qedhere
   \end{split}
\end{equation*}
\end{proof}

Having found the greatest lower bound for the families $\{\conv_\beta : \beta<\alpha\}$ for all limit $\alpha$, we now turn to the problem of finding the greatest lower bound for $\{\conv_\alpha : \alpha\in\omega_1\}$. 
We have not yet been able to determine it, nor to establish whether it exists at all, and thus the following question remains open.
\begin{question}
Does the family
$\{\conv_\alpha : \alpha\in\omega_1\}$
have the greatest lower bound in the Kat\v{e}tov order?
\end{question}

By Proposition~\ref{prop:hasldkfhas}
, the ideal $\conv$ is a lower bound of the family $\{\conv_\alpha : \alpha\in\omega_1\}$. We conjecture, however, that it is not the greatest lower bound. One could try to imitate the construction of the ideals $\conv_{<\alpha}$ in this case, however the main obstacle in this approach  
is that the family  of ideals which must be managed  has size $\omega_1$, whereas we  need an ideal defined on a countable set. A possible solution is suggested by the predecessor of the critical ideals $\conv_\alpha$, namely the ideal $\conv$. It is determined by an uncountable set $[0,1]$, but it is defined on its countable dense subset $\Q$. Accordingly, we pursue an analogous strategy and employ the theorem of Hewitt, Marczewski, Pondiczery (see e.g. \cite[Theorem
2.3.15]{MR1039321}) from  which it follows that the product of $\omega_1$ separable topological spaces is separable.   

\begin{definition}
    Let $D$ be a countable dense subset of the space 
    $$X = \prod_{2\leq \alpha < \omega_1} \omega^\alpha+1$$ and define the ideal
    $\conv_{\omega_1}=\{A\subset D: \forall \alpha\, \{f(\alpha) : f\in A\}\in\conv_\alpha\}.$
\end{definition}

\begin{proposition}
    $\conv_{\omega_1}\leq\conv_\alpha$ for each $\alpha$.
\end{proposition}
\begin{proof}
For each successor ordinal $\xi\in\omega^\alpha+1$ pick $d_\xi\in D$ with $d_\xi(\alpha)=\xi$. Define $f:\omega^\alpha+1\to D$ by 
\begin{equation*}
f(\xi) =
\begin{cases}
d_{\xi+1} & \xi \in \omega^\alpha,\\
d_1&  \xi= \omega^\alpha.
    \end{cases}
\end{equation*}
Then $f$ witnesses $\conv_{\omega_1}\leq\conv_\alpha$. To see this take any $A\notin\conv_\alpha$ and note that $f[A]\supset\{d_{\xi+1} : \xi\in A\setminus\{\omega^\alpha\}\}$. If we show that $B_\alpha=\{d_{\xi+1}(\alpha):\xi \in A\setminus\{\omega^\alpha\}\}\notin\conv_\alpha$ the proof will be finished, therefore it is enough to show that $A^d\subset B^d_\alpha$. Take any $\lambda\in A^d$. Then there is some increasing sequence $(\xi_n)_n$ with values in $A$ such that $\sup_n\xi_n=\lambda $. Then $\sup_n(\xi_n+1)=\lambda$ and $\xi_n+1=d_{\xi_n+1}(\alpha)\in B^d_\alpha $ for each $n$, hence $\lambda\in B^d_\alpha$. 
\end{proof}

\begin{question}
Is the ideal $\conv_{\omega_1}$
the greatest lower bound of the family 
$\{\conv_\alpha : \alpha\in\omega_1\}$
with respect to the Kat\v{e}tov order?
\end{question}


\section{Cardinal characteristics of critical ideals for countable compact spaces}
\label{subsec:cardinal-characteristics}
We now turn to the calculation of certain cardinal invariants of the ideals $\conv_\alpha$.

Some properties of ideals can be described by cardinal characteristics associated with them. There are four well known cardinal characteristics called additivity, covering, uniformity, and cofinality defined in the following way
for an ideal $\I$ on $X$ (see e.g.~\cite{MR1350295}):
$\add(\I) =\min\left\{|\cA|:\cA\subseteq\I\land \bigcup\cA\notin\I\right\}$,
$\cov(\I)  = \min\left\{|\cA|: \cA\subseteq\I\land \bigcup\cA=X\right\}$, 
$\non(\I)  = \min\{|A|:A\notin \I\}$, 
$\cof(\I)  =\min\{|\cA|:\cA\subseteq\I\land \forall B\in\I \, \exists A\in\cA \, (B\subseteq A)\}$.
These  characteristics are useful in the case of ideals on an uncountable set $X$ (for instance in the case of the $\sigma$-ideal  of all meager sets and the $\sigma$-ideal of all Lebesgue null sets). 
Apart from the cofinality, the other cardinal
characteristics are trivial, since they are all equal to $\aleph_0$.
Fortunately, 
Hern\'{a}ndez and Hru\v{s}\'{a}k
introduced in \cite{MR2319159} (see also~\cite{MR1686797} and \cite{MR2777744}) certain versions of these characteristics more suitable for tall ideals on countable sets:
\begin{equation*}
\begin{split}
\adds(\I) & =\min\{|\cA|:\cA\subseteq\I\land \neg \exists B\in\I \, \forall A\in\cA \, (A\subset^* B)\}, \\
\covs(\I)  & = \min\{|\cA|: \cA\subseteq\I\land \forall B\in [\omega]^{\aleph_0} \, \exists A\in\cA \, (|A\cap B|=\aleph_0)\},\\
\nons(\I) & = \min\{|\cA|:\cA\subseteq [\omega]^\omega\land \forall B\in\I \, \exists A\in\cA \, (|A\cap B|<\aleph_0)\}, \\
\cofs(\I) & =\min\{|\cA|:\cA\subseteq\I\land \forall B\in\I \, \exists A\in\cA \, (B\subset^* A)\}.
\end{split}
\end{equation*}

It is known (see e.g.~\cite{MR2777744}) that 
$\adds(\conv)=\adds(\Fin\otimes\Fin)=\aleph_0$,
$\nons(\conv)=\nons(\Fin\otimes\Fin)=\aleph_0$,
$\covs(\conv)=\cofs(\conv)=\continuum$, 
$\covs(\Fin\otimes\Fin)=\bnumber$ and  $\cofs(\Fin\otimes\Fin)=\dnumber$, where
$\bnumber$  is the \emph{bounding number} i.e.~the smallest cardinality of any unbounded family in the poset $(\omega^\omega,\leq^*)$ and 
$\dnumber$ is the \emph{dominating number} i.e.~the smallest cardinality of
any dominating family in the poset $(\omega^\omega,\leq^*)$ (for more on these cardinals see for instance \cite{MR2768685}). 

The following theorem provide the values of the above mentioned characteristics for critical ideals considered in the paper.
\begin{proposition}[\cite{critical-ideals-RF-MK-AK}] 
For every $\alpha$, 
$$\nons(\conv_\alpha) = \adds(\conv_\alpha)= \aleph_0, \quad 
\covs(\conv_\alpha)= \bnumber\quad\text{and} \quad \cofs(\conv_\alpha)=\dnumber.$$
\end{proposition}

There are also the following $\omega$-versions of the aforementioned cardinal characteristics, defined and studied in \cite{MR1686797}, \cite{Alek-Miller-Nulls}, \cite{Alek-Arturo-Miller-Trees}  and  \cite{MR4905423}:

\begin{equation*}
\begin{split}
\addo(\I) & =\min\{|\cA|:\cA\subseteq\I\land \neg \exists \cB\in[\I]^\omega \, \forall A\in\cA \ \exists B\in\mathcal{B} \, (A\subset^* B)\}, \\
\nono(\I) & = \min\{|\cA|:\cA\subseteq [\omega]^\omega\cap \I\land \forall \mathcal{B}\in[\I]^\omega \, \exists A\in\cA \, \forall B\in\mathcal{B} \, (|A\cap B|<\aleph_0)\}, \\
\cofo(\I) & =\min\{|\cA|:\cA\subseteq[\I]^\omega\land \forall \mathcal{B}\in[\I]^\omega \, \exists C\in\cA \, \forall B\in\mathcal{B} \, \exists A\in\mathcal{C} (B\subset^* A)\}
\end{split}
\end{equation*}
with the convention that $\min\emptyset=\continuum^+$.
Let us point out that whenever $\I$ is tall, then 
$$\nono(\I)  = \min\{|\cA|:\cA\subseteq [\omega]^\omega\land \forall \mathcal{B}\in[\I]^\omega \, \exists A\in\cA \, \forall B\in\mathcal{B} \, (|A\cap B|<\aleph_0)\}.$$

In \cite{Alek-Arturo-Miller-Trees} and  \cite{MR4905423} the authors carried out the following computations:
$\addo(\conv)=\nono(\conv)=\omega_1$,  
$\cofo(\conv)=\continuum$, 
$\addo(\Fin\otimes\Fin)=\bnumber$ and  $\nono(\Fin\otimes\Fin)=\cofo(\Fin\otimes\Fin)=\dnumber$.

Finally, we consider the cardinal invariant 
$$\covp(\I)= \min\{|\cA|: \cA\subseteq\I\land \forall B\in \I^+ \, \exists A\in\cA \, (|A\cap B|=\aleph_0)\},$$
introduced and studied in \cite{MR4472520}, where it was also shown that $\covp(\finxfin)=\omega$ and $\covp(\conv)=\continuum$.
There are several inequalities  among all the mentioned cardinal characteristics that hold for all tall ideals, as shown in Figure~\ref{fig:cardinal-characteristics}  (see e.g.~\cite[p.~578]{MR2777744} and  \cite{Alek-Miller-Nulls},  \cite{Alek-Arturo-Miller-Trees}). Moreover, if $\I\leq_{KB}\J$, then $\nono(\I)\leq\nono(\J)$  and if $\I\leq_{K}\J$, then $\covp(\I)\geq \covp(\J)$ (see \cite{Alek-Arturo-Miller-Trees} and \cite{MR4472520} resp.).

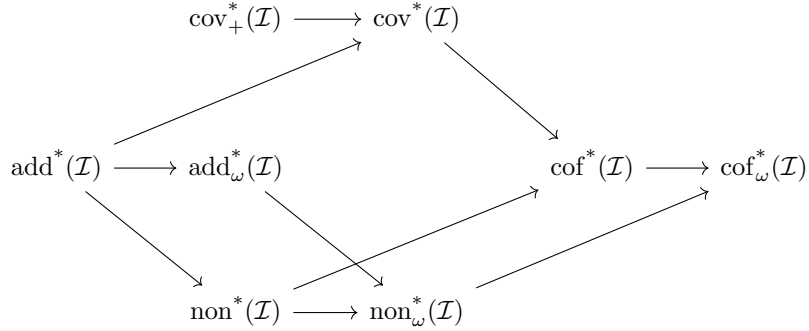
\begin{figure}
\centering

	\begin{tikzcd}
& & \covp(\I) \arrow[r] & \covs(\I) \arrow[rdd]
\\
\\
  & \adds(\I) \arrow[r] 
\arrow[rruu] \arrow[rdd] & \addo(\I) \arrow[rdd]
&&
\cofs(\I) \arrow[r] & \cofo(\I)  
\\
\\
&&\nons(\I) \arrow[r] \arrow[rruu] &\nono(\I) \arrow[rruu]
\end{tikzcd}
    \caption{Relationships between defined cardinal characteristics for tall ideals ($\kappa\to\lambda$ means $\kappa\leq \lambda$ in this diagram).}
    \label{fig:cardinal-characteristics}

\end{figure}

We say that a family $\cA$ witnesses $\addo(\I)$ whenever there is no $\cB\in[\I]^\omega$ such that for all  $ A\in\cA$ there is some  $B\in\cB$ with  $A\subset^* B$.
We say that a family $\cB$ witnesses  $|\cA|<\adds_\omega(\I)$ whenever for all $A\in\cA$ there is some $B\in\cB$ with $A\subset^*B$.
Analogous terminology applies to other cardinal invariants.

\begin{lemma}
    $$\addo(\{\emptyset\}\otimes\fin)= \bnumber
 \quad \text{and} \quad \cofo(\{\emptyset\}\otimes\fin)=\nono(\{\emptyset\}\otimes\fin) = \dnumber.$$
\end{lemma}

\begin{proof}
   
    Since the proofs of these equalities are rather similar, we present only the proof of the first. The others are obtained by the same method, with families defined analogously.
    
   $(\addo(\{\emptyset\}\otimes\fin)= \bnumber)$ Let $\cA$ be the family that witnesses $\addo(\fin\otimes\fin)= \bnumber$ and consider the family $\cG=\{A\setminus (i_A\times \omega) : A\in\cA \}\cap [\omega\times \omega]^\omega$, where $i_A=\min\{k: A\cap \{n\}\times\omega \text{ is finite for all } n>k\}+1$. Then $|\cG|\leq \bnumber$ and $\cG$ witnesses that $\addo(\{\emptyset\}\otimes\fin)\leq \bnumber$. Too see the other inequality it is enough to notice that the witness for $\addo(\{\emptyset\}\otimes\fin)$ also witnesses $\addo(\fin\otimes\fin)$.
\end{proof}

\begin{lemma}\label{lem:baufa}
$\conv_\alpha\approx \conv_\alpha\oplus\conv_\alpha$.
\end{lemma}
\begin{proof}
Define $f:\omega^\alpha\to \omega^\alpha\times 2$ by 
    \begin{equation*}
f(\xi)=
\begin{cases}
(\lambda+k/2,0) & k \text{ even},\\
(\lambda+(k-1)/2,1) & k \text{ odd},
    \end{cases}
\end{equation*}
where 
 $\lambda$ is 
  a unique limit ordinal and $k$ is a unique natural number  such that $\xi=\lambda+k$. 
Then $f$ is an isomorphism between $\conv(\omega^\alpha)$ and $\conv(\omega^\alpha)\oplus\conv(\omega^\alpha)$ and since $\conv_\alpha$ is isomorphic to $\conv(\omega^\alpha)$ the proof is finished.
\end{proof}

\begin{lemma}\label{lem:bhafdu}
    
    $ \conv(\omega^\alpha\cdot n)\approx\conv_\alpha$, for every $\alpha$ and every $n\in\omega\setminus\{0\}$.
\end{lemma}

\begin{proof}
By induction on $n\in\omega\setminus\{0\}$. Suppose the statement holds for $n$. Notice that $\conv(\omega^\alpha\cdot (n+1))\approx\conv(\omega^\alpha\cdot n)\oplus\conv([\omega^\alpha\cdot n, \omega^\alpha\cdot (n+1)))$, as witnessed by $f:\omega^\alpha\cdot (n+1)\to (\omega^\alpha\cdot (n+1) )\times 2$, 
\begin{equation*}
f(\xi) =
\begin{cases}
(\xi,0) & \xi \in \omega^\alpha\cdot n,\\
(\xi,1) & \xi \in [\omega^\alpha\cdot n, \omega^\alpha\cdot (n+1)).
    \end{cases}
\end{equation*}
Using the induction hypothesis together with the fact that $\conv_\alpha\approx \conv([\omega^\alpha\cdot n, \omega^\alpha\cdot (n+1)))$, it is enough to show that $\conv_\alpha\approx \conv_\alpha\oplus\conv_\alpha$, which is true by Lemma~\ref{lem:baufa}.
\end{proof}

\begin{proposition} 
For every $\alpha$, 
$$\addo(\conv_\alpha)= \bnumber
 \quad \text{and} \quad \cofo(\conv_\alpha)=\nono(\conv_\alpha) = \dnumber.$$
\end{proposition}

\begin{proof}
    
    ($\addo(\conv_\alpha)\leq \bnumber$)  
    Since $\addo(\conv_2)= \addo(\finxfin)=\bnumber$,
    it is enough to show 
 $\addo(\conv_\alpha)\leq \addo(\conv_2)$ for each $\alpha$. Let $\cA\subset \conv_2$ be a witness for $\addo(\conv_2)$. 
 The proof will be finished, once we show that $\cA$ is also 
 a witness for $\addo(\conv_\alpha)$. 
 Take any $\{B_n : n\in \omega\}\subset \conv_\alpha$. Then $\{B_n\cap(\omega^2+1): n\in \omega\} \subset \conv_2 $, so there is some $A\in \cA$ such that $A\not\subset^* B_n\cap (\omega^2+1) $ for each $n$.  Since $\cA\subset \cP(\omega^2+1)$, then $A\not\subset^* B_n $ for each $n$ as well.

      ($\addo(\conv_\alpha)\geq \bnumber$)  
    We proceed by induction on $\alpha$. Take any $\alpha$ and assume that the inequality  holds for any $\beta < \alpha$.  
    In the case that $\alpha=\beta+1$ for some $\beta$, define $\lambda_n=\omega^\beta\cdot n$. In the case that $\alpha$ is a limit ordinal, take a strictly increasing sequence $(\alpha_n)_{n\in\omega}$ that is cofinal in $\alpha$ with $\alpha_0=0$ and define $\lambda_n=\omega^{\alpha_n}$.
     Take any family $\cA\subset\conv_\alpha$ such that $|\cA|<\bnumber$. For every $F\in\cA$ define a natural number $k_F$ as $$k_F=\min\{k\in\omega : \forall n\geq k \ ( |F\cap[\lambda_n, \lambda_{n+1}]|<\omega)\}. $$ By Proposition~\ref{prop:gwiazdka}, $k_F$ is well-defined. Define also families $$\cA_0=\{F \cap [0 , \lambda_{k_F}) : F\in\cA\} \quad \text{and} \quad  \cA_1=\{F\setminus [0,\lambda_{k_F}) : F\in \cA\}.$$ 
     Consider the ideal $\J=\{C\subset \omega^\alpha : \forall n \ (|C\cap [\lambda_n, \lambda_{n+1}]< \omega )\}.$ Note that $\J$ is isomorphic to the ideal $\{\emptyset\}\otimes\fin$ and $\J\subset \conv_\alpha$. Since $\addo(\{\emptyset\}\otimes\fin)=\bnumber$,  $|\cA_1|<\bnumber$ and $\cA_1\subseteq\cJ$,  we obtain that $|\cA_1|<\addo(\J)$. Therefore, there exists a family $\mathcal{Y}\in [\J]^\omega$ such that for every $A_1\in\cA_1$ there is some $Y\in \mathcal{Y}$ such that $A_1\subset^*Y$. 

     For every $n\in \omega $, notice that $\cG= \{F_0\cap[\lambda_n, \lambda_{n+1}] : F_0\in\mathcal{F}_0\} \subset \conv([\lambda_n, \lambda_{n+1}])$ and  $|\cG|< \bnumber$.
     By induction hypothesis,  $|\cG|<\addo(\conv_\beta)= \addo(\conv([\lambda_n, \lambda_{n+1}]))$ in the successor case and 
     $|\cG|<\addo(\conv_{\alpha_{n+1}})= \addo(\conv([\lambda_n, \lambda_{n+1}]))$ in the limit case. In both cases, we obtain a family  $\mathcal{X}_n\in [\conv([\lambda_n , \lambda_{n+1}])]^\omega$ such that for every $A_0\in\cA_0$ there is  $X\in\mathcal{X}_n$ such that $A_0\cap[\lambda_n , \lambda_{n+1}]\subset^*X$.

      Consider the family $$\cB=\left\{\bigcup s\cup Y : s\in\left[\bigcup_{n\in\omega} \mathcal{X}_n\right]^{<\omega}, Y\in\mathcal{Y}\right\}\in [\conv_\alpha ]^\omega. $$ We claim that it witnesses $|\cA|<\adds_\omega(\conv_{\alpha})$. To see this, take any $F\in\cA$ and the corresponding $k_F$. Then there is some $Y\in\mathcal{Y}$ such that $F\setminus [0,\lambda_{ k_F})\subset^*Y$ and for every $n\leq k_F$  there is some $X_n\in\mathcal{X}_n$ such that $F\cap[\lambda_n, \lambda_{n+1})\subset^*X_n$. Then $F\subset^*\bigcup_{n\leq k_F}X_n\cup Y$.

    ($\cofo(\conv_\alpha)=\dnumber$) 
     For each $\alpha$, $\dnumber=\cofs(\conv_\alpha)\leq\cofo(\conv_\alpha)$. 

     Now we prove that $\dnumber \geq \cofo(\conv_\alpha)$ by induction on $\alpha$. Assume that it holds for $\beta < \alpha$.
     In the case that $\alpha=\beta+1$ for some $\beta$, define $\lambda_n=\omega^\beta\cdot n$. In the case that $\alpha$ is a limit ordinal, take a strictly increasing sequence $(\alpha_n)_{n\in\omega}$ that is cofinal in $\alpha$ with $\alpha_0=0$ and define $\lambda_n=\omega^{\alpha_n}$.

     For each $n\geq1$, using Lemma \ref{lem:bhafdu} together with the induction hypothesis, let $$\mathcal{F}_n=\left\{\left\{F^n_{(\gamma,k)}:k\in\omega\right\} : \gamma\in \dnumber \right\} \subset [\conv(\lambda_n)]^\omega$$ be an enumeration of a witness for $\cofo(\conv(\lambda_n))=\dnumber$. Consider the ideal $\mathcal{J}=\{A\subset \omega^\alpha : \forall n \ (|A\cap [\lambda_n, \lambda_{n+1}]< \omega )\}.$ Note that $\mathcal{J}$ is isomorphic to the ideal $\{\emptyset\}\otimes\fin$ and that $\mathcal{J}\subset \conv_\alpha$.  Let $$\mathcal{K}=\{\{K_{(\delta,j)}:j\in\omega\} : \delta\in \dnumber\} \subset [\J]^\omega$$ be an enumeration of a witness for $\cofo(\J)=\dnumber$.  Define the family $$\cA=\{\{F^n_{(\gamma,k)}\cup K_{(\delta,j)}:j,k,n\in\omega\} : \gamma, \delta\in\dnumber\}.$$ Then $\cA\subset[\conv_\alpha]^\omega$, $|\cA|\leq\dnumber$, and we claim that it witnesses for every family $\{A_i :i\in\omega\}\subset \conv_\alpha$ there is some $\gamma_0$ and $\delta_0$ such that for every $i\in\omega $ there is some $n,j,k\in\omega$ such that $A_i\subset^*F^n_{(\gamma_0,k)}\cup K_{(\delta_0,j)}$.
     To see this, take any  family $\{A_i :i\in\omega\}\subset \conv_\alpha$ and define a function $f:\omega\to\omega$ by $f(i)=\min\{k:\forall n\geq k \ (|A_i\cap [\lambda_n,\lambda_{n+1})|<\omega)\}+1$. For the family $\{A_i\setminus[0, \lambda_{f(i)}) : i\in\omega\}\subset \J$ there is some $\delta_0$ such that for each $i$ there is some $j_i$ such that $A_i\setminus[0, \lambda_{f(i)})\subset^*K_{(\delta_0,j_i)}$. For each $n$,  there is some $\gamma_n$ such that for each $i$ there is some $k_i$ such that $A_i\cap [0, \lambda_n)\subset^*F^n_{(\gamma_n,k_i)}$. Then the family $\{F^n_{(\gamma_n,k_i)}\cup K_{(\delta_0,j_i)}:i,n\in\omega\}\in \cG$ is as required. Indeed, take any $A_i$ and the corresponding $f(i)$. Then $A_i\cap [0,\lambda_{f(i)})\subset^*F^{f(i)}_{(\gamma_{f(i)},k_i)}$ and $A\setminus [0,\lambda_{f(i)})\subset^*K_{(\delta_0,j_i)}$, and therefore $A_i\subset^*F^{f(i)}_{(\gamma_{f(i)},k_i)}\cup K_{(\delta_0,j_i)}$.

     ($\nons_\omega(\conv_\alpha) = \dnumber$) 
     Since $\conv_\alpha \leq_{K}  \conv_2 \approx \Fin\otimes\fin$, we obtain $\nono (\conv_\alpha)\leq \nono(\Fin\otimes\fin)=\dnumber$ for each $\alpha$.
    Now we prove that  $\nono (\conv_\alpha)\geq \dnumber$ by induction on $\alpha$. Take any $\alpha$ and assume that it holds for $\beta < \alpha$.     
    In the case that $\alpha=\beta+1$ for some $\beta$, define $\lambda_n=\omega^\beta\cdot n$. In the case that $\alpha$ is a limit ordinal, take a strictly increasing sequence $(\alpha_n)_{n\in\omega}$ that is cofinal in $\alpha$ with $\alpha_0=0$ and define $\lambda_n=\omega^{\alpha_n}$.
     Take any family $\cA\subset [\omega^\alpha]^\omega$ such that $|\cA|<\dnumber$. Define the family $$\cA_0=\{A\in\cA : \forall n \ (A\cap[\lambda_n, \lambda_{n+1})<\omega)\}.$$
     Consider the ideal $\J=\{C\subset \omega^\alpha : \forall n \ (|C\cap [\lambda_n, \lambda_{n+1}]< \omega )\}.$ Note that $\J$ is isomorphic to the ideal $\{\emptyset\}\otimes\fin$ and $\cJ\subset \conv_\alpha$. Then $\cA_0\subset \J\cap[\omega^\alpha]^\omega$, and since $|\cA_0|<\dnumber$ and $\nono(\J)=\dnumber$,  there exists $\cB_0\in [\J]^\omega$ such that for every $A\in\mathcal{A}_0$ there is some $B\in\cB_0$ such that the set $B\cap A$ is infinite.

      For every $n\geq1$, consider the family $$\cA_n= \{A\cap\lambda_n : A\in\mathcal{A}\text{ and } |A\cap\lambda_n|=\omega\}.$$ Then $\mathcal{A}_n\subset[\lambda_n]^\omega$ and since $|\mathcal{A}_n|<\dnumber$, by the induction hypothesis and the Lemma \ref{lem:bhafdu}, $|\cA_n|<\nono(\conv(\lambda_n))$ so there exists $\mathcal{B}_n\in[\conv(\lambda_n)]^\omega$ such that for every $A\in\mathcal{F}_n$ there is some $B\in\mathcal{B}_n$ such that the set $A\cap B\cap\lambda_n$ is infinite. 
      Let $\mathcal{B}=\bigcup_{n\in\omega}\mathcal{B}_n $. Then $\cB\in [\conv_\alpha]^\omega$ and we claim that it witnesses $|\mathcal{A}|<\nono(\conv_{\alpha})$. Take any $A\in\mathcal{A}$. Since it is infinite, either there is some $n\geq 1$ such that the set $A\cap\lambda_n$ is infinite, then the set $A\cap B$ is infinite for some $B\in\mathcal{B}_n$, or $A\in\cA_0$ in which case one can find a corresponding $B\in\cB_0$ with $
      B\cap A$ being infinite.
\end{proof}

Up to this point, all cardinal invariants of $\conv_\alpha$ have taken the same values as for $\fin\otimes\fin$. We now turn to the first invariant where they differ, namely $\covp(\I)$. This difference is unavoidable, since this invariant is closely related to the property $P^-(\omega)$, as shown in \cite[Theorem~6.5 and Proposition~2.5]{MR4584767}. From this it follows that its value for $\conv_\alpha$ for $\alpha\geq3$ must be uncountable, whereas for $\fin\otimes\fin$ it is countable, which we will not prove, since we determine the exact value below.

\begin{proposition}\label{prop:covp}
      For every $\alpha\geq3$,
    $$\covp(\conv_\alpha)=\bnumber.$$

\end{proposition}

\begin{proof}
First we show that $\covp(\conv_3)=\bnumber.$
    One inequality follows from the fact that $\covp(\conv_3)\leq\covs(\conv_3)=\bnumber$. To see the reversed inequality, take any family $\cA\subset \conv_3$ such that $|\cA|<\bnumber$. For every $j\in\omega$ and $F\in\cA$ consider the function $f_{F,j}:\omega\to\omega $ defined by 
    $$f_{F,j}(i)=\max\{k\in\omega: \omega^2\cdot j+\omega \cdot i+k\in F\} $$ 
    with the convention that maximum of an infinite and empty set is $0$. 
    As $|\{f_{F,j} : F\in\cA , j\in\omega\}|\leq |\cA|\cdot \omega<\bnumber$, there is some function $f:\omega\to\omega$ such that $f_{F,j}\leq^*f$ for each $F\in\cA$ and $j\in\omega$. Define 
    $$A=\{\omega^2\cdot j+\omega\cdot i+(f(i)+1): i,j\in\omega\}.$$
    Since $\cF\subset\conv_3$, for every $F\in\cA$ and every $j\in\omega$ there is some $i(F,j)\in\omega$ such that the set $F\cap [\omega^2\cdot j +\omega\cdot i, \omega^2\cdot j+\omega\cdot (j+1)]$ is finite for each $i\geq i(F,j)$. Therefore the set $A\cap F\cap [\omega^2\cdot j, \omega^2\cdot (j+1)]$ is finite for every $F\in\cA$ and $j\in\omega$.
    For every $F\in\cA$ define $g_F:\omega\to\omega$ by 
    $$g_F(j)=\min\{k : \forall i\geq k \ (A\cap F\cap [\omega^2\cdot j+\omega \cdot i, \omega^2 \cdot j +\omega\cdot (i+1))=\emptyset)\}.$$
    
    Again,  $|\{g_{F} : F\in\cA \}|\leq |\cA|<\bnumber$, so there is some function $g:\omega\to\omega$ such  that $g_F\leq^*g$ for each $F\in\cA$. Define $$B=A\setminus \bigcup_{j\in\omega}[\omega^2\cdot j , \omega^2\cdot j +\omega\cdot g(j)).$$ Then $B\not\in \conv_3$ and we claim that $B$ witnesses $|\cA|<\covp(\conv_3)$. To see this take any $F\in\cA$ and notice that since  there is some $j_F$ such that $g_F(j)\leq g(j)$ for every $j\geq j_F$, we obtain that $B\cap F\cap [\omega^2\cdot j +\omega\cdot g(j) , \omega^2\cdot (j+1)] =\emptyset$ for all $j\geq j_F$. And since for every $j<j_F$, we know that $B\cap F\cap [\omega^2\cdot j, \omega^2\cdot (j+1))$ is finite, the set $B\cap F$ is finite as well.

    We now address the general case. Assume $\alpha\geq 3$. Using the properties of these cardinal invariants, we know that  $\covp(\conv_\alpha)\leq\covs(\conv_\alpha)=\bnumber$ and $\covp(\conv_\alpha)\geq\covp(\conv_3)=\bnumber$ since  $\conv_\alpha\leq_K\conv_3$.
\end{proof}


\bibliographystyle{amsplain}
\bibliography{references}

\end{document}